\documentclass[final,12pt]{elsarticle}
\usepackage{physics}
\usepackage{amsmath}
\usepackage{amsfonts}
\usepackage{amsthm}
\usepackage{fancybox}
\usepackage{amssymb}
\graphicspath{ {figures/} } 
\usepackage{subcaption}
\usepackage{xurl}
\usepackage[ruled,vlined,lined,onelanguage,resetcount,algosection]{algorithm2e} 

\usepackage[english]{babel}

\newtheorem{theorem}{Theorem}[section]
\newtheorem{lemma}[theorem]{Lemma}

\newtheorem{definition}[theorem]{Definition}
\newtheorem{example}[theorem]{Example}

\newtheorem{remark}[theorem]{Remark}

\newtheorem{main problem}[theorem]{Main problem}

\newtheorem{corollary}[theorem]{Corollary}

\usepackage[many]{tcolorbox}
\usepackage{multicol}
\usepackage{xcolor,colortbl}
\usepackage{longtable}
\usepackage{multirow}

\tcbset {
  base/.style={
    arc=0mm, 
    bottomtitle=0.5mm,
    boxrule=0mm,
    colbacktitle=black!10!white,
    colback=black!3!white,
    coltitle=black, 
    fonttitle=\bfseries, 
    left=2.5mm,
    leftrule=1mm,
    right=3.5mm,
    title={}, 
    toptitle=0.75mm, 
  }
}

\definecolor{brandblue}{rgb}{0.34, 0.7, 1}
\newtcolorbox{mainbox}[1]{
  colframe=black!20!white, 
  base={#1}
}

\begin{document}

\begin{frontmatter}



\title{An optimal convergent Collatz algorithm}


\author[inst1]{J.C. Riano-Rojas}

\affiliation[inst1]{organization={Departamento de Matemáticas y Estadística, Universidad Nacional de Colombia},
            addressline={jcrianoro@unal.edu.co}, 
            city={Manizales},
            state={Caldas},
            country={Colombia },
            orcidID = {orcid:0000-0002-5719-2854}            
            }

\begin{abstract}
In this research, an optimal algorithm for the Collatz conjecture is presented. Properties such as the convergence of the algorithm and an equation that relates the algorithm to the classical Collatz conjecture are obtained. It is validated that the proposed theory is correct with several examples; as a sector of mathematicians believes that the Collatz conjecture fails in some power of $3$. The algorithm was applied to the first $600$ powers of $3$, where the convergence of the proposed algorithm was verified.

\end{abstract}




\end{frontmatter}



\section{Introduction}

The Collatz conjecture is still an interesting problem in mathematics. Several hundred articles have been devoted to it, as can be observed in Lagarias \citep{Lag1},\citep{Lag2},\citep{Lag3}, and \citep{Lag4}; where there is a state of the art on the subject up to $2009$, mentioning the different interesting properties of the Collatz function; as well as various generalizations; or relationships with other types of problems not only in the mathematical, but also in the computational field, and how some of them have proven the undecidability of the conjecture. That is why this conjecture becomes a challenge that forces mathematicians to try to decipher it. In many of the articles mentioned in the Lagarias review, there is no explicit proposal for an algorithm like the one formulated in this research, much less the proof of its convergence.

The charm of the conjecture has not escaped the great mathematician Tao in \citep{Tao}. He showed that, for any function $f : \mathbb{N} + 1 \longrightarrow \mathbb{R}$ with $\lim_{N \longrightarrow \infty} f(N) = +\infty $, one has $Col_{min(N)} \leq f(N)$, for almost all $N \in \mathbb{N}+1 $(in the sense of logarithmic density). The proof proceeds by establishing a stabilization property for a certain first passage random variable; associated with the Collatz iteration (or more precisely, the closely related Syracuse iteration); which, in turn, is derived from an estimate of the characteristic function of a certain skew random walk on a 3-adic cyclic group $\mathbb{Z}/3^{n}\mathbb{Z}$ at high frequencies. This estimate is achieved by studying how a certain two-dimensional renewal process interacts with a union of triangles associated with a given frequency. Although his advances are important, they are not related to the focus of this research.

Barina presents at \citep{Barina}, a new algorithmic approach for computational convergence verification of the Collatz problem. The main contribution of the paper is the replacement of huge precomputed tables containing $O(2^N)$ entries with small lookup tables comprising only $O(N)$ elements. Our single-threaded CPU implementation can check $4.2 \times 10^9$ 128-bit numbers per second on a computer with an Intel Xeon Gold $5218$ CPU but does not prove the conjecture.

The most recent research; on the conjecture was carried out by Heule et all at \citep{Heule}. They explored the Collatz conjecture and its variants through the lens of termination of
string rewriting. They built a rewriting system that simulates the iterated application of the Collatz function on strings, corresponding to mixed binary–ternary representations of positive integers. They showed that; the termination of this rewriting system is equivalent to the Collatz conjecture. Furthermore, they proved that a previously studied rewrite system; that simulates the Collatz function using unary representations, does not admit termination proofs by natural matrix interpretations, even when it is used in conjunction with dependency pairs. To show the feasibility of this approach in proving mathematically interesting claims, they also implemented a minimal termination prover using natural/arctic matrix interpretations and they found automated proofs of nontrivial weakening of the Collatz conjecture. Although they fail to prove the Collatz conjecture, they believe that the ideas here represent an interesting new approach. However, it is different from the algorithm proposed here.

The proposed algorithm\ref{alg1} is shown in section \ref{secc2} of this document and its operation is verified with example \ref{exam1}. The convergence of the algorithm is tested in section\ref{secc3}, and two results are included that can be useful to determine the undecidability of the Collatz conjecture. The article ends with  section\ref{secc4} of conclusions.

\section{Optimal Collatz algorithm}\label{secc2}

It is introduced by recalling notations, definitions, and examples necessary to understand the algorithm proposed in this document; which are found in the literature; but, to unify concepts, they are replicated here.

\subsection{Collatz Function}

The original Collatz function was case-defined as follows.
 
\begin{definition}\label{def1}
Let $n \in \mathbb{N}^{+}$ be arbitrary, construct
\begin{equation}
C(n) = \left\{
\begin{array}
[c]{l}%
\frac{n}{2}, \text{ if } n\equiv 0 \text{ (mod 2)}\\
3n+1, \text{ if } n \equiv 1  \text{ (mod 2)} 
\end{array}
\right. \label{ecua1}%
\end{equation}

where, $C$ is called the Collatz function. $\mathbb{N}^{+} = \mathbb{N} - {\{0\}}$; and $\equiv  \text{ (mod 2)}$ is congruence modulo 2. 

Let $n \in \mathbb{N}^{+}$, and let $k \in \mathbb{N}$ be denoted by
$C^{k}(n) = \underbrace{C \circ \cdots \circ C}_{k \text{-times}}(n)$, for a fixed $n$. The set of images of compositions $C^{k}$ is denoted by
\begin{equation}
\mathbb{C}( n) = \lbrace m \in \mathbb{N}^{+}: (\exists k \in \mathbb{N})(m = C^{k}(n))\rbrace.\label{ecua2}%
\end{equation}

\end{definition}

These concepts are shown in the following example\ref{exam1}.

\begin{example}\label{exam1}
Considering $n = 106$, and $k=10$, in table \ref{tabla1}, it can be seen that, $C^{10}(106)=4$ is in oval column 11 since it starts from column 0. \\

\begin{table}[ht]
\centering
    \begin{tabular}
    {|>{\columncolor{gray!30}}p{0.65cm} |p{0.365cm} |p{0.365cm} |p{0.365cm} |p{0.365cm} |p{0.365cm} |p{0.365cm} |p{0.365cm} |p{0.365cm} |p{0.365cm} |p{0.365cm} |p{0.365cm} |p{0.365cm} |p{0.365cm} |p{0.365cm} |p{0.365cm} |p{0.365cm} |}
    \rowcolor{gray!30}
    \hline
    \textbf{k} & \multicolumn{1}{l|}{\textbf{0}} & \textbf{1} &  \textbf{2} & \textbf{3} & 
    \textbf{4} & \textbf{5} & \textbf{6} & \textbf{7} & \textbf{8} & \textbf{9} & \textbf{10} & \textbf{11} & \textbf{12} & \textbf{13} & \textbf{14} & \textbf{15}   \\ \hline
\footnotesize{$C^{k}(n)$} & $106$ & $53$ & $160$ & $80$ & $40$ & $20$ & $10$ & $5$ & $16$ & $8$ & \cellcolor{gray!10}{$\Ovalbox{\textbf{4}}$} & $2$ & $1$ & $4$ & $2$ & $1$ \\ \hline
\end{tabular}
\caption{Collatz functions $C^{k}$ to $n = 106$}\label{tabla1}
\end{table}

Therefore, $\mathbb{C}(36)=\lbrace 106,53,160,80,40,20,10,5,16,8,4,2,1 \rbrace$, then $\vert\mathbb{C}(36)\vert = 12$ is finite.
\end{example}

Based on the above notion, Collatz formulated main problem is the conjecture \ref{main:prob1} below; which is still considered an open problem in Mathematics.

\begin{main problem}[Collatz conjecture]\label{main:prob1}

\begin{equation}
\forall n \in \mathbb{N}^{+} \exists k \in \mathbb{N} \left( C^{k}(n) = 1 \right).\label{ecua3}%
\end{equation}
An equivalent version is,

\begin{equation}
\forall n \in \mathbb{N}^{+} \mathbb{C}(n) \textit{ is finite}.\label{ecua4}%
\end{equation}

\end{main problem}

The following functions, proposed in the definition \ref{def2}, are introduced to simplify the writing.

\begin{definition}\label{def2}
Let $k \in \mathbb{N}^{+}$, be the discrete derivative of $C^{k}(n)$, respect to $k$ is defined by $\frac{\partial C^{k}(n)}{\partial k}= C^{k+1}(n) -C^{k}(n)$, which is denoted by $\Delta_{k}$.
\end{definition}

It is intended to analyze the behavior of the function $C^{k}(n)$ to identify patterns that can characterize its behavior. When calculating the discrete derivative of $C^{k}(n)$, the following theorem is central to the algorithm proposed in this article.

\begin{theorem}\label{teo1}
Let $n \in \mathbb{N}^{+}$ and let $k \in \mathbb{N}^{+}$, $C^{k}(n)\equiv 1   \text{ (mod 2)}$; if, and only if, $\Delta_{k}>0$. And $C^{k}(n)\equiv 0   \text{ (mod 2)}$, if and only if $\Delta_{k}<0$.
\end{theorem}

The proof can be seen in the \ref{teo1A}.

\begin{example}\label{exam2}
Consider $n = 106$,\\
\begin{table}[ht]
\centering
\begin{tabular}
    {|>{\columncolor{gray!30}} p{0.63cm} |p{0.51cm} |>{\columncolor{gray!10}} p{0.51cm} |p{0.51cm} |p{0.51cm} |p{0.51cm} |p{0.51cm} |p{0.51cm} |>{\columncolor{gray!10}} p{0.51cm} |p{0.51cm} |p{0.51cm} |p{0.51cm} |p{0.51cm} |>{\columncolor{gray!10}} p{0.51cm} |}

\rowcolor{gray!30}
\hline
\textbf{k} & \multicolumn{1}{l|}{\textbf{0}} & \textbf{1} &  \textbf{2} & \textbf{3} & 
    \textbf{4} & \textbf{5} & \textbf{6} & \textbf{7} & \textbf{8} & \textbf{9} & \textbf{10} & \textbf{11} & \textbf{12}  \\ \hline

\hline
\footnotesize{$C^{k}(n)$} & $106$ & $\textbf{53}$ & $160$ & $80$ & $40$ & $20$ & $10$ & $\textbf{5}$ & $16$ & $8$ & $4$ & $2$ & $\textbf{1}$ \\
\hline
$\Delta_{k}$ & $-53$ & $\Ovalbox{\textbf{107}}$ & $-80$ & $-40$ & $-20$ & $-10$ & $-5$ & $\Ovalbox{\textbf{11}}$ & $-8$ & $-4$ & $-2$ & $-1$ & $\Ovalbox{\textbf{3}}$  \\
\hline

\end{tabular}
\caption{$C^{k}(n)$ functions and $\Delta_{k}$  to $n = 106$}\label{tabla2}
\end{table}
\end{example}

\subsection{Optimal Collatz algorithm}

The definition of two functions used in algorithm \ref{alg1} in which; the decomposition into prime factors of a natural number is required, is presented.

\begin{definition}\label{def3}
Let $n \in \mathbb{N}^{+}$, taking the decomposition into prime factors of the natural numbers $n =\prod_{i=0}^{k_{n}} p_{i}^{\alpha_{i}}$, where the $p_{i}$ are prime numbers ordered in increasing order, with $\alpha_{i},k_{n} \in \mathbb{N}$. It will be set for construction that, $ p_{0} = 2$, is defined for $\mathbb{E}(n) = p_{0}^{\alpha_{0}}$, the power $2$ greater than possess $n$; too it is defined by $\mathbb{O}(n) = \prod_{i = 1}^{k_{n}}p_{i}^{\alpha_ {i}}$ to the odd part greater than possess $n$.
\end{definition}

\begin{example}\label{exam2b}
Let the following two natural numbers be:
\begin{description}
 \item[a.] Let $n = 3200$ its decomposition into prime factors is $n = 2^{7}5^{2}$, then $\mathbb{E}(n) = 2^{7} = 128$ and $ \mathbb{O}(n) = 5^{2} = 25$.
 \item[b.] Let $n = 12782924$ its decomposition into prime factors is $n = 2^{2}7^{4}11^{3}$, then $\mathbb{E}(n) = 2^{2} = 4$ and $\mathbb{O}(n) = 7^{4}11^{3} = 3195731$.
\end{description}    
\end{example}

\begin{remark}\label{remark2}
Notice these properties on $\mathbb{E}(n)$ and $\mathbb{O}(n)$ that are evident to show. It turns out that; $\mathbb{O}^{2}(n) = \mathbb{O}(n)$ and $\mathbb{E}^{2}(n) = \mathbb{E}(n)$. This is because these two functions extract the maximum even and odd components of $n$. Furthermore, it happens that $\mathbb{E}\circ \mathbb{O}(n) = 1$ and $\mathbb{O}\circ \mathbb{E}(n) = 1$.
\end{remark}

For simplicity and to unify notation, comment \ref{remark3}, contains the variables that were used in the algorithm and how they will be used in proofs developed in this document. The proposed algorithm \ref{alg1} was implemented in Matlab2020.

\begin{remark}\label{remark3}
The variables of algorithm \ref{alg1} are sequences. Then, it indices $x(i)=x_{i}, \ldots , w(i) = w_{i}$ are used, ($i$ is often used simply). For example, in line $16$ of the algorithm \ref{alg1}, the expression $u_{i} \gets \log_{2}(\mathbb{E}(x_{i}))$ it is represented.
\end{remark}

\begin{algorithm}[ht]
\caption{Optimal Collatz algorithm}\label{alg1}
\KwIn{\\
$n \gets input(\text{"Enter an natural number:  "})$ \\ 
}
\KwOut{$w$,$x$,$y$,$z$,$u$,$v$\;
}
$i = 0$\\
$w(i) \gets i$\\
$x(i) \gets n$\\
$y(i) \gets \mathbb{O}(n)$\\
$z(i) \gets \mathbb{E}(n)$\\
$u(i) \gets log_{2}(\mathbb{E}(n))$\\
$v(i) \gets 0$\\
$flag \gets 1$\\

\While{$flag \neq 0$}{
	$i \gets i + 1$\;
	$w(i) \gets i + 1$\;
	$x(i) \gets C(y(i-1))$\;
	$y(i) \gets \mathbb{O}(x(i))$\;
	$z(i) \gets \mathbb{E}(x(i))$\;
	$u(i) \gets \log_{2}(\mathbb{E}(x(i)))$\;
	$Dist \gets distance([x(i+1),y(i)];[x(i),y(i-1)])$\;
	$v(i-1) \gets Dist$\;
  	\If {$Dist==0$}{
		$flag \gets 0$\;
  	}
}
\Return{[$w$,$x$,$y$,$z$,$u$,$v$]}
\end{algorithm}

Figure \ref{fig1} shows the behavior of algorithm \ref{alg1}, where each curve corresponds to the value of $y$ for the natural numbers from $1$ to $200$. The convergence of all the curves to the value $1$ is observed as the values on the $w$ axis move away. In addition, the values of the sequences generated by the algorithm, for the natural numbers that start with $n = 54$ and $n = 167$ are shown.

\begin{figure}[ht!]
\centering
\caption{Collatz optimal algorithm behavior of $w$ vs $y$ for values between $1-200$}\label{fig1}
\includegraphics[width=9cm,height=7cm]{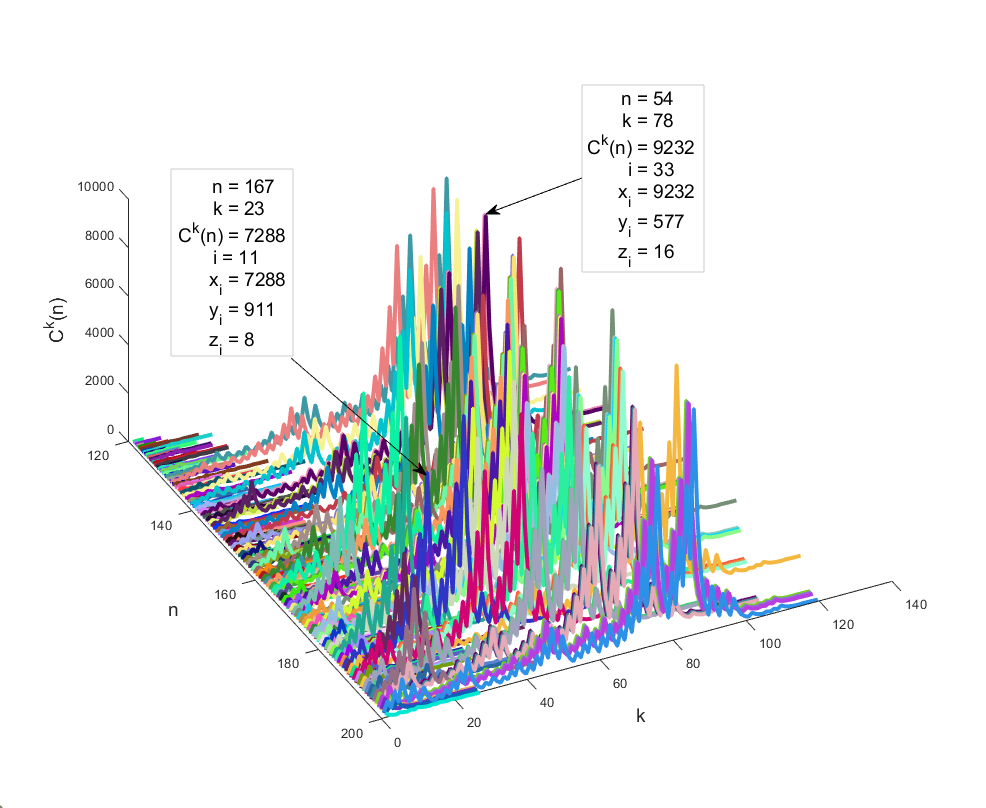}
\subcaption{3D representation}
\includegraphics[width=11cm,height=5cm]{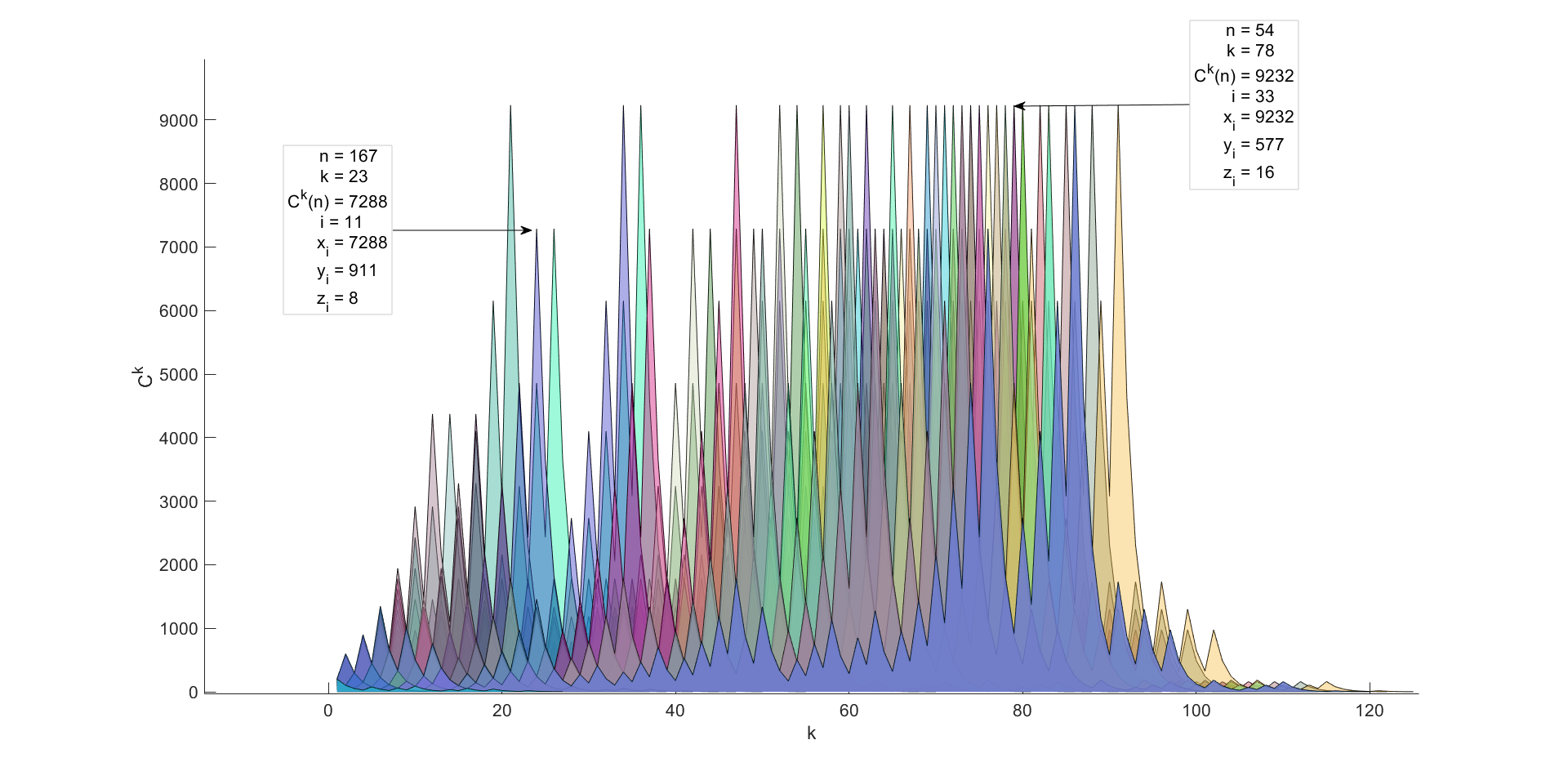}
\subcaption{Projection in the $k$,$C^{k}$ plane}
\end{figure}

Figure \ref{fig2} shows the behavior of the Collatz function, $C$ is in defined \ref{def1} as a graph, a neighborhood of node 1 was amplified.

\begin{figure}[ht!]
\includegraphics[scale=0.95]{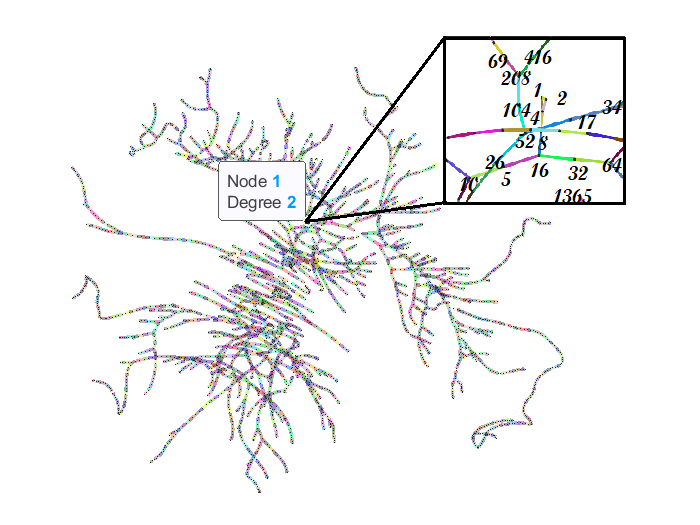}
\caption{Graph of Collatz function between $1-1500$}\label{fig2}
\end{figure}
 
Figure \ref{fig3} shows the behavior of the Optimal Collatz algorithm\ref{alg1}. The neighborhood of node 1 was amplified, to show the behavior of some neighboring nodes. It is seen how the graph in the figure is simplified. \ref{fig2}.

\begin{figure}[ht!]
\includegraphics[scale=0.7]{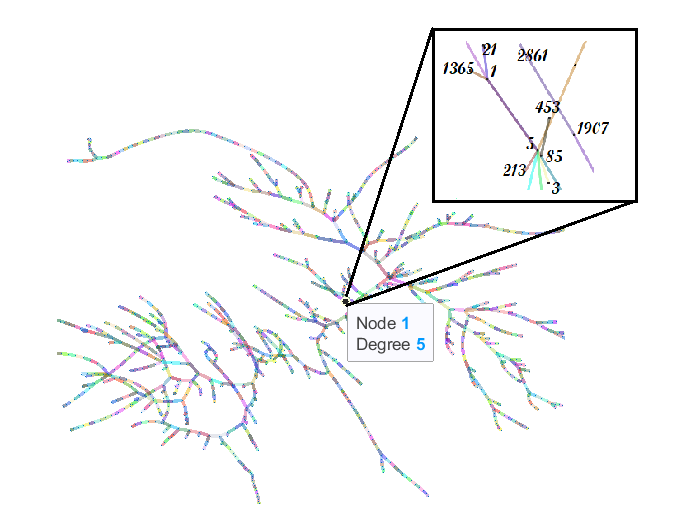}
\caption{Graph of optimal Collatz algorithm between $1-1500$}\label{fig3}
\end{figure}

A fourth presentation of the Collatz sequences can be seen in the appendix \ref{fig4a}, the sequences converge to the axis parallel to the z axis, translated to the point (4,1,0). The x axis represents the sequences $x_{k}$, the y axis represents the sequence $y_{k}$ generated by the algorithm \ref{alg1}, on the z axis the sequence $C^{k}(n)$ $n$ for each natural number  $1 \leq n \leq 200$
 
\section{Convergence of the proposed optimal Collatz algorithm}\label{secc3}

The most interesting results of this document are presented in this section, stating the basic theoretical elements that are required for the proofs of the lemma \ref{lema1}, and the corollaries \ref{coro1}- \ref{coro3}.

\subsection{Fundamental Property of Diophantine Equations}

The Diophantine equations are recalled in theorem \ref{teo2}. It is a fundamental topic of number theory and a result that is used to guarantee that the couple $(x = 4, y = 1)$ belongs to the solution of the Diophantine equation that is obtained for any natural number $m$ to which the algorithm\ref{alg1} is applied.

\begin{theorem}\label{teo2}
Let $a,b,c \in \mathbb{Z}$ be the equation

\begin{equation}
ax + by = c.\label{ecua6}%
\end{equation}
it has infinitely many solutions, if and only if,
$\alpha = \gcd(a,b)$ divides $c$, where $\gcd(a,b)$ is the greatest common divisor of $a$, $b$.
Furthermore, if a particular solution $x = s$, and $y = t$ is known, the general solution is of the form:

\begin{equation}
\left\{
\begin{array}
[c]{l}%
x = s + \eta \frac{b}{\alpha} \\
y = t -\eta\frac{a}{\alpha}
\end{array}
\right. \label{ecua7}%
\end{equation}
where $\eta \in \mathbb{Z}$
\end{theorem}

The proof can be consulted in Zukerman et al. \citep{Niven}. One consequence is corollary \ref{coro1}, related to the Collatz function.
 
\begin{corollary}\label{coro1}
Let $C:\mathbb{N} \to \mathbb{N}$ be the Collatz function and then; $x = C(y)$ is solvable. To simplify the proof, it is divided into the following cases:
\begin{itemize}
\item[a.] If $y \equiv 1  \text{ (mod 2)}$ then $x = C(y)$ is solvable.
\item[b.] If $y \equiv 0  \text{ (mod 2)}$ then $x = C(y)$ is solvable.
\end{itemize}
\end{corollary}

The proof can be seen in the \ref{coro1A}.

The corollary \ref{coro2} presents a relationship between $\eta$ and $y$, that is used in algorithm\ref{alg1}.

\begin{corollary}\label{coro2}
Let $C:\mathbb{N} \to \mathbb{N}$ be the Collatz function; then, so that, for all $i \in \mathbb{N}$ then, $x_{i+1} = C(y_{i})$  algorithm \ref{alg1} is solvable, and the general solution is of the form:
 
\begin{equation}
\left\{
\begin{array}
[c]{l}%
x_{i+1} = x_{1} - 3 \eta_{i} \\
y_{i} = y_{0} - \eta_{i}.
\end{array}
\right. \label{ecua8}%
\end{equation}
where $\eta_{i} \in \mathbb{Z}$.
  
In addition, from the equations \eqref{ecua8} it follows that:
\begin{equation}
\eta_{i}=y_{0}-y_{i}\label{ecua9}%
\end{equation}

\end{corollary}
The proof can be seen in the \ref{coro2A}.
\begin{remark}
It is observed that the couple $(4,1)$ satisfies the Diophantine equation $x - 3y = 1$ independent of the particular solution $(x_{1},y_{0})$ used by the algorithm \ref{alg1}. In example \ref{exam3}, it is shown how algorithm \ref{alg1} calculates the values of $x_{i},y_{i}...$ for $ n = 3200$, $x_{1},y_{0}$, and $\eta_{i}$ in \eqref{ecua9}.    
\end{remark}

\begin{example}\label{exam3}
Consider $n = 3200$, the values for the variables $x_{i},y_{i},z_{i},u_{i}, \eta_{i} $\\ are calculated in the table \ref{tabla3}.
In row $i = 7$ of table \ref{tabla3} the algorithm \ref{alg1} obtains the following results: $x_{i} = 16$, $y_{i} = \mathbb{O}(16) = 1 $, $y_{i} = \mathbb{E}(16) = 2^{4}$, $u_{i} = log_{2}(16) = 4$ and $\eta_{i} = 24$. Moreover, it will be observed that the sequences $x_{i} \rightarrow 4$, $y_{i} \rightarrow 1$, $\eta_{i} \rightarrow 24$, and $v_{i} \rightarrow 0$ when $i \rightarrow \infty$.

\begin{table}[t]
 \centering
\begin{tabular}{|c|c|>{\columncolor{gray!10}}c|c|c|c|c|}
\hline
\rowcolor{gray!30}
\hline
$w$ & $x=C(\mathbb{O}(y))$ & $y=\mathbb{O}(x)$ & $z=\mathbb{E}(x)$ & $u=\log_{2}(z)$ & $\eta_{i}$ & $v=Dist$ \\ 
\hline 
$0$ & $3200$ & $\textbf{25}$ & $2^{7}$ & $7$ & $0$ & $18.9$ \\ 
\hline 
$1$ & $76$ & $\textbf{19}$ & $2^{2}$ &  $2$ & $6$ & $31.6$ \\ 
\hline 
$2$ & $58$ & $\textbf{29}$ & $2^{1}$ &  $1$ & $-4$ & $56.9$ \\ 
\hline 
$3$ & $88$ & $\textbf{11}$ & $2^{3}$ &  $3$ & $14$ & $18.9$ \\ 
\hline 
$4$ & $34$ & $\textbf{17}$ & $2^{1}$ &  $1$ & $8$ & $12.6$ \\ 
\hline 
$5$ & $52$ & $\textbf{13}$ & $2^{2}$ &  $2$ & $12$ & $25.2$ \\ 
\hline 
$6$ & $40$ & $\textbf{5}$ & $2^{3}$ &  $3$ & $20$ & $12.6$ \\ 
\hline
\rowcolor{gray!10}
$\textbf{7}$ & $16$ & $\Ovalbox{\textbf{  1  }}$ & $2^{4}$ &  $4$ &  $\Ovalbox{\textbf{  24  }}$ & $\textbf{0}$ \\ 
\hline 
$8$ & $\Ovalbox{\textbf{  4  }}$ & $\textbf{1}$ & $2^{2}$ &  $2$ & $24$ & $0$ \\ 
\hline
$9$ & $4$ & $\textbf{1}$ & $2^{2}$ &  $2$ & $24$ & $0$ \\ 
\hline
$\vdots$ & $\vdots$ & $\vdots$ & $\vdots$ &  $\vdots$ & $\vdots$ & $\vdots$ \\ 
\hline 
$i$ & $4$ & $\textbf{1}$ & $2^{2}$ &  $2$ & $24$ & $0$ \\ 
\hline 
$\vdots$ & $\vdots$ & $\vdots$ & $\vdots$ &  $\vdots$ & $\vdots$ & $\vdots$ \\ 
\hline 
\end{tabular}
\caption{Collatz optimal algorithm behavior  to $n = 3200$}\label{tabla3}
\end{table}
\end{example}

\subsection{Proof of the convergence of Collatz's optimal algorithm}
The central idea of the convergence proof of the proposed algorithm \ref{alg1} is to analyze the behavior of the elements close to the minimum in the sequence $y_{i}$ for some $m \in \mathbb{N}^{ +}$.

\begin{lemma}[Convergence of Collatz's optimal algorithm]\label{lema1}
\[\forall m \in \mathbb{N}^{+}\exists i \in \mathbb{N}(y_{i} = 1)\]
\end{lemma}

\begin{proof}
\textbf{[Reductio ad absurdum]} Suppose there exists $m^{*} \in \mathbb{N}^{+}$, so that, for every $i \in \mathbb{N}$, so that $y_{i} \neq 1$, then constructing the set \[S_{m^{*}} = \lbrace z \in \mathbb{N} : \exists i \in \mathbb{N} (z = y_{i} \neq 1), y_{i} \text{ generated by algorithm~\ref{alg1}, } y_{0} = \mathbb{O}(m^{*}) \rbrace.\] It follows that $S_{m^{*}} \neq \emptyset$, because $y_{0} \in S_{m^{*}}$. Then, by the well order of the natural numbers $S_{m^{*}}$ has a minimum. That is, $\exists y_{i_{*}} \in \mathbb{N} ( y_{i_{*}} = min \left( S_{m^{*}}\right) )$ then, $y_{i_{*}} \neq 1$ is generated by algorithm \ref{alg1}.  Taking $y_{i_{*}} - 1$ it is know that is an even number it follows, $y_{i_{*}} - 1 \not\in S_{m^{*}}$, applying $\mathbb{O}$, so that $\mathbb{O}(y_{i_{*}} - 1)=y_{*}$, also, $y_{*}\not\in S_{m^{*}}$, thus, implying two cases: either $y_{*} = 1$, or $y_{*} \neq 1$ and is not generated by algorithm \ref{alg1}.

\textbf{Case1:} If $y_{*} = 1$, then $\mathbb{E}(y_{i_{*}} - 1)=2^{\alpha}$ with $\alpha>1$. Applying $\mathbb{E}$ to both sides, it follows that $\mathbb{E}^{2}(y_{i_{*}} - 1) = \mathbb{E}(2^{\alpha})=2^{\alpha}$. But, by comment \ref{remark2}, $\mathbb{E}^{2}=\mathbb{E}$, thus $\mathbb{E}(y_{i_{*}} - 1) = y_{i_{*}} - 1$. Solving for $ y_{i_{*}}$, so that $y_{i_{*}} = 2^{\alpha} + 1$, with $\alpha>1$. Applying $C$ in the last expression, we get $C(y_{i_{*}}) = 3(2^{\alpha} + 1) + 1 = 3\cdot2^{\alpha} + 4 =2^{2}(3\cdot2^{\alpha-2}+1)$. But $C( y_{i_{*}}) = x_{i_{*} + 1}$, so $x_{i_{*} + 1} = 2^{2}(3\cdot2^{\alpha-2}+1)$ and now, applying $\mathbb{O}$ to both sides, it is had $\mathbb{O}(x_{i_{*} + 1})= 3\cdot2^{\alpha-2}+1$, but $\mathbb{O}(x_{i_{*} + 1})= y_{i_{*} + 1}$. Therefore, $y_{i_{*} + 1} = 3\cdot2^{\alpha-2}+1$. By hypothesis, it follows that $y_{i_{*}} < y_{i_{*} + 1}$, since $y_{i_{*}} = min(S_{m^{*}})$. Substituting $y_{i_{*}}$ and $y_{i_{*} + 1}$ for their respective values, the following expression is obtained, $2^{\alpha} + 1 < 3\cdot2^{\alpha-2}+1$, simplifying, so that $4 < 3$, which is absurd.

\textbf{Case2:} If $y_{*} \neq 1$ and it is not generated by algorithm \ref{alg1}, on the other hand, taking $\mathbb{E}(y_{i_{*}} - 1) = 2^{\alpha}$ with $\alpha \geq 1$, then $y_{i_{*}} = 2 ^{\alpha}y_{*} + 1$, two cases occur: $\alpha > 1$ or $\alpha = 1$.

\textbf{Case2a:} If $\alpha > 1$ similar to case 1. Applying $C$, it is had $x_{i_{*}+1} = 3*2^{\alpha}y_{i_{*}} + 4$, then by algorithm \ref{alg1}. It is apply $\mathbb{O}$, it is obtained that $y_{i_{*}+1} = 3*2^{\alpha - 2}y_{*} + 1$, but for hypothesis $y_{i_{*}+1} \geq y_{i_{*}}$, solving and simplifying it follows that $3*2^{\alpha - 2}y_{*} \geq 2^{\alpha}y_{*}$ which is absurd.

\textbf{Case2b:} If $\alpha = 1$ and considering $y_{i_{*} - 1} = 2^{\beta}\mu + 1$ with $\mu \equiv 1$ (mod2) and $\mu$ is the maximum decomposition of power primes non two. There are four cases to analyze $\beta = 1$ or $\beta = 2$ or $\beta = 3$ or $\beta > 3$.

\textbf{Case2b.1:} If $\beta = 1$ it is considers $y_{i_{*} - 1} = 2\mu + 1$. By algorithm \ref{alg1} then, $y_{i_{*}} = \mathbb{O}( C(y_{i_{*} - 1}))= 2y_{*} + 1$ substituting and simplifying. It follows that $3\mu + 2 = 2y_{*} + 1$, where it is obtained that $3\mu + 1 = y_{i_ {*}} - 1$. Substituting and simplifying, we have $2\mu +1 \leq y_{i_{*}}$ which means that $y_{i_{*} - 1} \leq y_{i_{*}}$. Absurd, since contradicts that $y_ {i_{*}}$ is minimum of $S_{m^{*}}$.

\textbf{Case2b.2:} If $\beta = 2$ it is considers $y_{i_{*} - 1} = 2^{2}\mu + 1$. By algorithm \ref{alg1} then, $y_{i_{*}} = \mathbb {O}(C(y_{i_{*} - 1}))= 2y_{*} + 1$ substituting and simplifying. It follows that $3\mu + 1 = 2y_{*} + 1$, where simplifying it is obtained that $3\mu + 1 = y_{i_{*}}$. Absurd, since $y_{i_{*}}$ is an odd number and $3\mu + 1$ is an even number.

\textbf{Case2b.3:} If $\beta = 3$ it is considers $y_{i_{*} - 1} = 2^{3}\mu + 1$. By algorithm \ref{alg1} then, substituting and simplifying. It follows that $x_{i_{*}} - 1  = 8y_{*} + 3$, $x_{i_{*}+1} - 1  = 6y_{*} + 3$, and $x_{i_{*}+2} - 1  = 9y_{*} + 6$. Absurd, it contradicts that they must all be multiples of $3$, when $y_{*}$ is not a multiple of $3$, and when $y_{*}$ is a multiple of $3$ it is also absurd; because it contradicts that $y_{i_{*}+2}$ is odd.

\textbf{Case2b.4:} If $\beta > 3$ it is considers $y_{i_{*} - 1} = 2^{\beta}\mu + 1$. By algorithm \ref{alg1} then, $y_{i_{*}} = \mathbb {O}(C(y_{i_{*} - 1}))= 2y_{*} + 1$ substituting and simplifying. It follows that $3*2^{\beta-2}\mu + 1 = 2y_{*} + 1$, where simplifying it is obtained that $3*2^{\beta - 3}\mu = y_{*}$. Absurd, since $y_{*}$ is an odd number and $3*2^{\beta - 3} \mu $ is an even number.

Hence, it is concluded that $\forall m \in \mathbb{N}^{+}\exists i \in \mathbb{N}(y_{i} = 1)$.
\end{proof}

\begin{remark}\label{remark4}
As a challenge, the reader can seek to reduce the number of cases in the previous proof.
\end{remark}

A corollary \ref{coro3} is presented that shows a relationship between the optimal algorithm proposed and its convergence with the conjecture \ref{main:prob1}.

\begin{corollary}\label{coro3}
\begin{equation}
\forall n \in \mathbb{N}^{+} \mathbb{C}(n) \textit{ is finite}.%
\end{equation}

\end{corollary}

\begin{proof}
Let $n \in \mathbb{N}^{+}$, construct the sequence $\lbrace y_{i} \rbrace _{i \in \mathbb{N}} $ by the algorithm \ref{alg1} where $y_{0} = \mathbb{O}(n)$, is convergent by the lemma \ref{lema1}. Then, exists $ i_{*} \in \mathbb{N}$, such that $y_{i_{*}} = 1$; then, constructing the set:
\[ S = \lbrace i \in \mathbb{N} : y_{i} = 1, y_{i} \text{ generated by the algorithm\ref{alg1}} \rbrace.\] Then, $S \neq \emptyset$, since $i_{*} \in S$, by well order of the naturals, exist $i_{min} \in \mathbb{N}^{+}$ such that $i_{min} = min(S)$, also for all $i \in S$, such that, $i>i_{min}$, it is had $y_{i}=1$, applying $C$ to this last expression, it is obtained that, $C(y_{i})=4$, an expression that when applying $\mathbb{O}$, so that $\mathbb{O}(C(y_{i}))=\mathbb{O}(4)=1$, then $y_{i+1} = 1$. Therefore, $Image(\mathbb{C}(n))$ is finite; since it is the union of the finite images of $y_{i+1}\neq 1$, joined with $1$ and joined with all the images of applying $C$ when $C^{k}(n)$ is even (that are known when the derivative $\Delta_{k}$ is negative), and that is calculated by the algorithm \ref{alg1} with the variable $u = log_{2}(\mathbb{E})$. This can be summarized in the following equation \ref{ecua10}.

\begin{equation}
\vert \mathbb{C}(n) \vert = i_{min} + \sum_{i=0}^{i_{min}} log_{2}(\mathbb{E}(x_{i})) + 1.\label{ecua10}
\end{equation}

\end{proof}

\begin{example}\label{exam4}
  For $n = $3200, in table \ref{tabla3}. The gray row shows the step where the algorithm proposed in the $y_{i}$ column reaches the value 1 for the first time. For this reason, $i_{min} = 7$, adding the values of the column $u=log_{2}(x_{i})$, $23$ is obtained, and applying the formula \eqref{ecua10}, which was derived from the corollary \ref{coro3}, $\vert\mathbb{C} \vert= 31$ is obtained, which coincides with the application of the Collatz functions $C$. That is, $C^ {31}(3200)$; is a term in the sequence generated by the Collatz function \eqref{ecua1} that reaches $1$ for the first time.
  
\end{example}

The formula \eqref{ecua10} was implemented in the proposed algorithm \ref{alg1}. The results were verified with all the runs that were made. Furthermore, a part of the mathematical community believes that the conjecture is false and that it fails in some power of $3$, in figure \ref{fig5} the opposite is seen, because when calculating the algorithm \ref{alg1} on the first 600 powers of 3, the sequences $x_{i}$ and $y_{i}$ converge respectively to $4$ and $1$, all of them approach the axis parallel to the z-axis centered at $(4,1,0)$.

\begin{figure}[ht!]
\centering
\includegraphics[width=13cm,height=10cm]{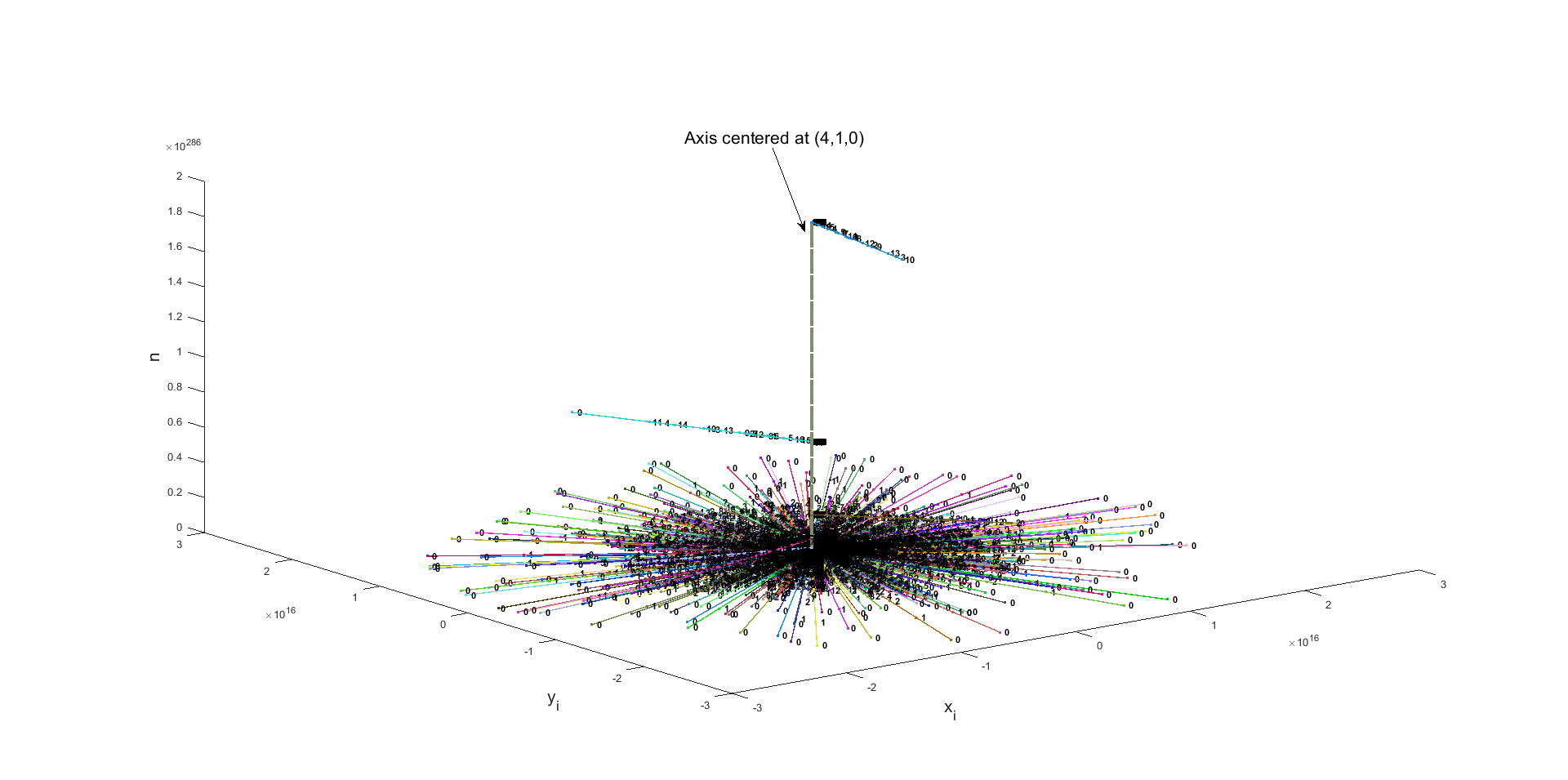}
\caption{Collatz optimal algorithm behavior of $x$ vs $y$ for values power of $3$ between $1-600$.}\label{fig5}
\end{figure}

\section{Conclusions}\label{secc4}

In this research, an optimal algorithm \ref{alg1} for the Collatz function was proposed, decomposing each term of the sequence into two factors, a power factor of $2$ that is calculated with the function $\mathbb{E}$ and the other non-power factor of $2$ that is calculated with the function $\mathbb{O}$. The proof of their convergence was carried out. Several examples were presented that validate the proofs.

A relationship is established between the equation \eqref{ecua10}, which is obtained as a consequence of the algorithm \ref{alg1} with the classical Collatz function. The corollary \ref{coro3} is an important proof of the possible convergence of the classical Collatz conjecture.

\subsection*{Acknowledgements}
The author want to thank MINCIENCIAS for the financial support through the project Plataforma de analítica de datos como soporte a la gestión de la demanda operacional en la industria manufacturera de la red de aliados de Revlog, approved in call 890 of 2020. The authors also want to thank the Universidad Nacional de Colombia - Manizales, for allowing the spaces for this research.

\appendix

\section*{Appendix Section}
\label{sec:appendix}

\begin{theorem}\label{teo1A}
Let $n \in \mathbb{N}^{+}$ and let $k \in \mathbb{N}^{+}$, $C^{k}(n)\equiv 1   \text{ (mod 2)}$; if, and only if, $\Delta_{k}>0$. And $C^{k}(n)\equiv 0   \text{ (mod 2)}$, if and only if $\Delta_{k}<0$.
\end{theorem}

\begin{proof}
Let $n \in \mathbb{N}^{+}$; let $k \in \mathbb{N}^{+}$, and suppose $C^{k}(n) \equiv 1  \text{ (mod 2)}$. Then $C^{k+1}(n)= 3C^{k}(n)+1$; by direct calculation of $\Delta_{k}=(3C^{k}(n)+ 1) - C^{k}(n) = 2C^{k}(n) + 1$. Now, suppose $C^{k}(n)\equiv 0  \text{ (mod 2)}$, then $C^{k+1}(n) = \frac{C^{k}(n)}{2}$, by direct calculation of $\Delta_{k} = \frac{C^{k}(n)}{2}-C^{k}(n) = -\frac{C^{k}(n)} {2}$.
\end{proof}

\begin{corollary}\label{coro1A}
Let $C:\mathbb{N} \to \mathbb{N}$ be the Collatz function and then; $x = C(y)$ is solvable. To simplify the proof, it is divided into the following cases:
\begin{itemize}
\item[a.] If $y \equiv 1  \text{ (mod 2)}$ then $x = C(y)$ is solvable.
\item[b.] If $y \equiv 0  \text{ (mod 2)}$ then $x = C(y)$ is solvable.
\end{itemize}
\end{corollary}

\begin{proof}
First let's show (a.) supposing that $y \equiv 1  \text{ (mod 2)}$, then; $x = C(y) = 3y + 1$, equivalent to the Diophantine equation $x - 3y = 1$ is solvable, since $a=1$, $b=-3$ and $c=1$, this fulfilling the conditions of theorem \ref{teo2}. Furthermore, by algorithm \ref{alg1} proposed, taking $s=x_{1}$ and $t=y_{0}$, it is a particular solution of $x - 3y = 1$, since $x_{1} = C(y_{0})$. Therefore, the general solution is $x = s -3\eta$ and $y = t -\eta$.

To prove (b.) suppose that $y \equiv 0  \text{ (mod 2)}$, then; $x = C(y) = \frac{y}{2}$, equivalent to the Diophantine equation $2x - y = 0$ is solvable, since $a=2$, $b=-1$ and $c=0$ meet the conditions of theorem\ref{teo2}.
\end{proof}

\begin{corollary}\label{coro2A}
Let $C:\mathbb{N} \to \mathbb{N}$ be the Collatz function; then, so that, for all $i \in \mathbb{N}$ then, $x_{i+1} = C(y_{i})$  algorithm \ref{alg1} is solvable, and the general solution is of the form:
 
\begin{equation}
\left\{
\begin{array}
[c]{l}%
x_{i+1} = x_{1} - 3 \eta_{i} \\
y_{i} = y_{0} - \eta_{i}.
\end{array}
\right. \label{ecua8A}%
\end{equation}
where $\eta_{i} \in \mathbb{Z}$.
  
In addition, from the equations \eqref{ecua8A} it follows that:
\begin{equation}
\eta_{i}=y_{0}-y_{i}\label{ecua9A}%
\end{equation}

\end{corollary}

\begin{proof}

Equation\eqref{ecua8A} is a direct consequence of corollary \ref{coro1A}.
To show the equation \eqref{ecua9A}, the first equality of \eqref{ecua7} was taken and, applying the definition\eqref{ecua1} we get $3y_{i} + 1 = 3y_{0} + 1 - 3 \eta_{i}$. Then; adding this expression with the second expression of \eqref{ecua8A}, we get $4y_{i} + 1 = 4y_{ 0} + 1 - 4 \eta_{i}$, simplifying and solving for $\eta_{i}$, the expression \eqref{ecua9A} is derived.
\end{proof}

A fourth presentation of the Collatz sequences can be seen in the following figure \ref{fig4a}, the sequences converge to the axis parallel to the z axis, translated to the point (4,1,0). The x axis represents the sequences E $x_{k}$, the y axis represents the sequence E $y_{k}$ generated by the algorithm \ref{alg1}, on the z axis the sequence $C^{k}(n)$ for each natural number  $1 \leq n \leq 200$.
The numbers indicate the order $k$ in the sequence $(x_{k+1},y_{k},n)$.

\begin{figure}[b]
\centering
\caption{Collatz optimal algorithm behavior of $x$ vs $y$ for values between $1-200$}\label{fig4a}
\includegraphics[width=9cm,height=7cm]{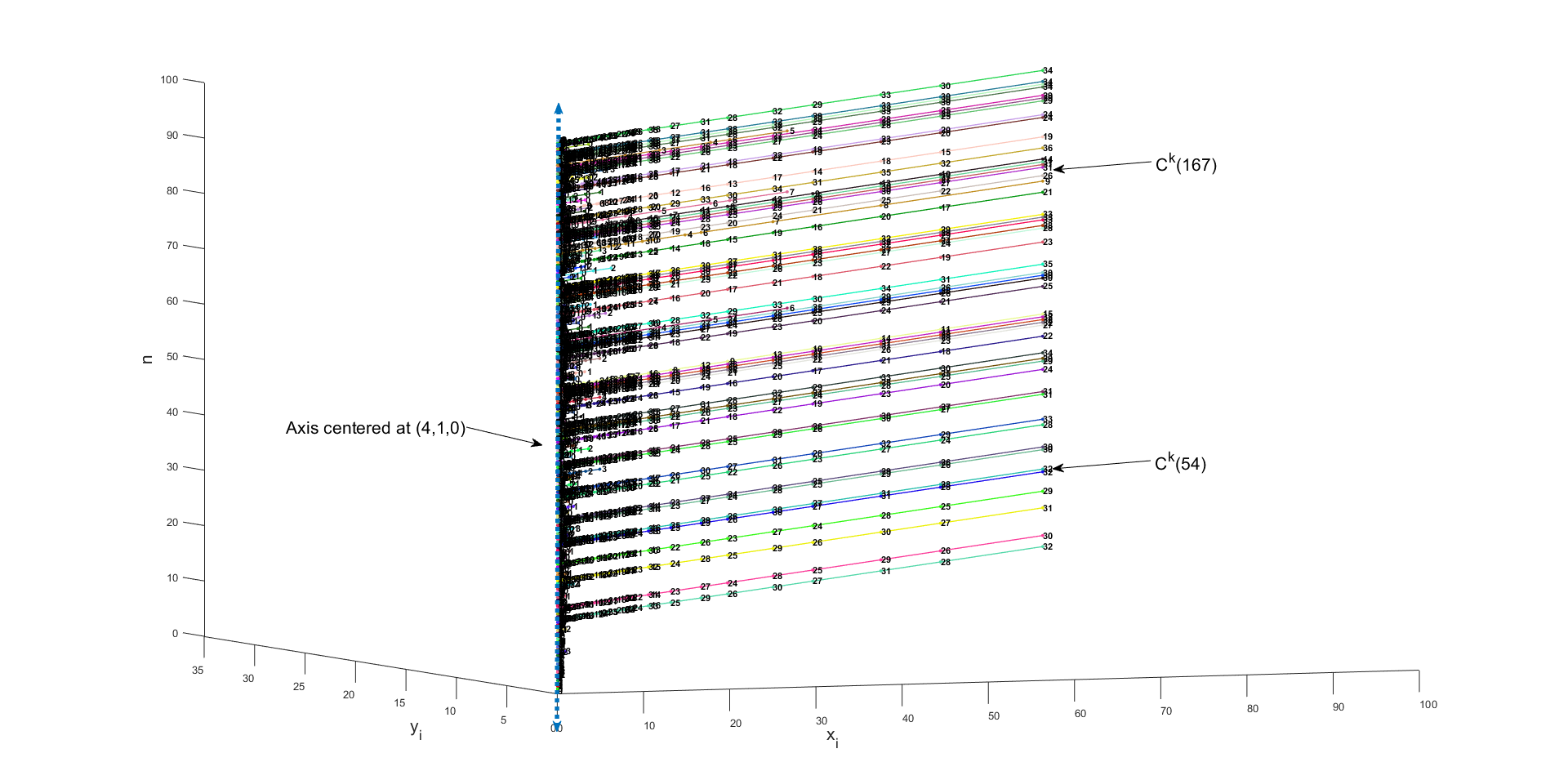}
\subcaption{3D representation}
\includegraphics[width=9cm,height=9cm]{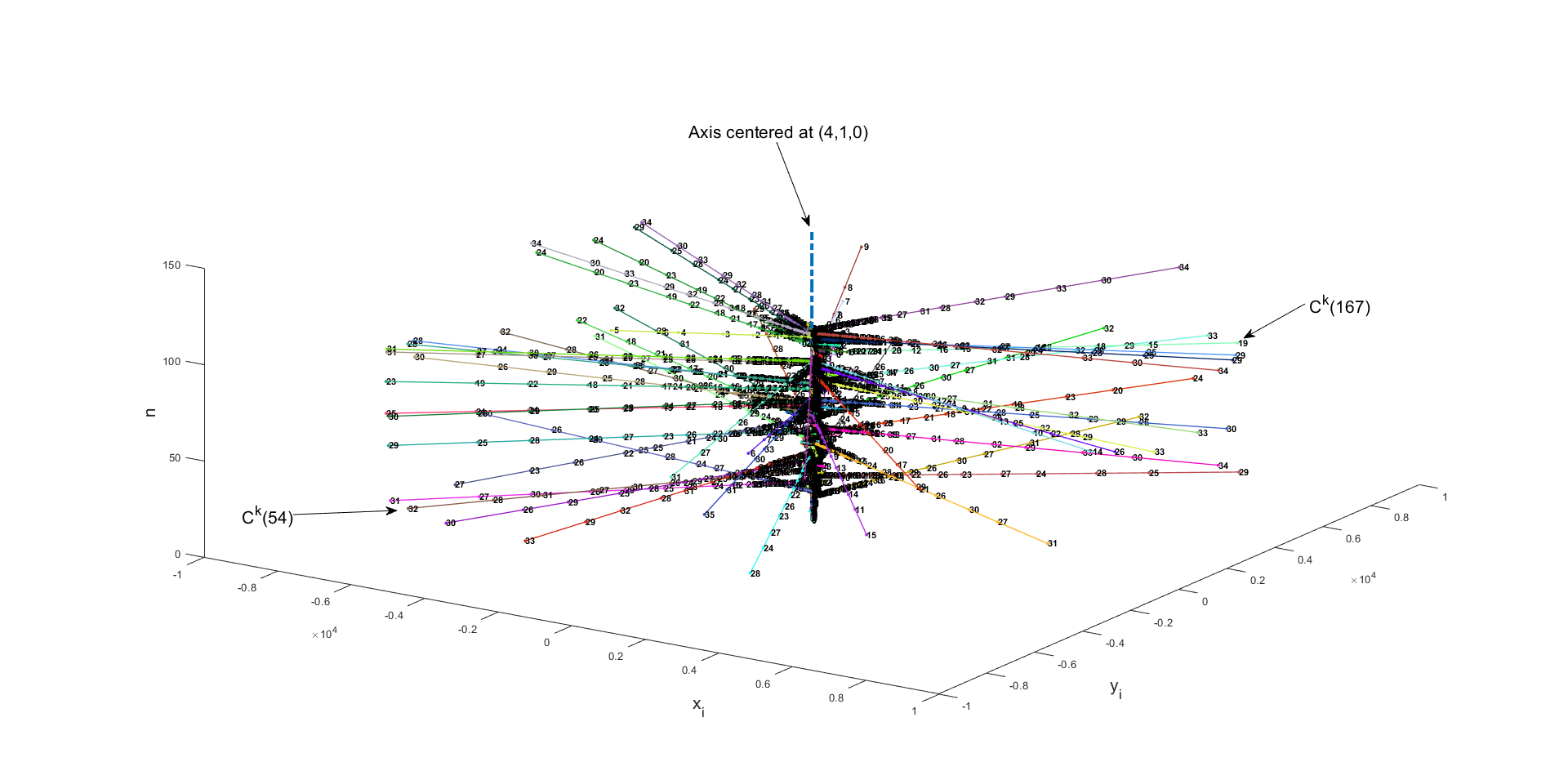}
\subcaption{The sequences were randomly rotated to improve the vision of the representation of figure (4a).}
\end{figure}

\bibliographystyle{amsplain}

\end{document}